\documentclass[12pt,reqno]{amsart}

\usepackage{hyperref}
\usepackage[usenames]{color}
\usepackage[latin1]{inputenc} 
\usepackage{float}
\usepackage{bbm}

\usepackage[displaymath, mathlines]{lineno}

\usepackage[english]{babel} 

\usepackage{amssymb, amsmath}
\usepackage{geometry,graphicx}
\usepackage{cite}
\usepackage{enumitem}

\usepackage{xurl}

\allowdisplaybreaks

\theoremstyle{plain}
\newtheorem{Theorem}{Theorem}[section] 

\newtheorem{Proposition}[Theorem]{Proposition}

\newtheorem{Remark}[Theorem]{Remark}

\newtheorem{Example}[Theorem]{Example}
\newtheorem{conjecture}[Theorem]{Conjecture}

\numberwithin{equation}{section} 
\begin{document}

\begin{center}
 \textbf{Theory on new fractional operators using normalization and probability tools}
\end{center}

\begin{center}
 Marc Jornet
\end{center}

\begin{center}
Departament de Matem\`atiques, Universitat de Val\`encia, 46100 Burjassot, Spain. \\
email: marc.jornet@uv.es \\
ORCID: 0000-0003-0748-3730
\end{center}

\ \\
\textit{Reference:} Fractal Fract. 2024, 8(11), 665; https://doi.org/10.3390/fractalfract8110665

\ \\
\textbf{Abstract.} We show how a rescaling of fractional operators with bounded kernels may help circumvent their documented deficiencies, for example, the inconsistency at zero or the lack of inverse integral operator. On the other hand, we build a novel class of linear operators with memory effects to extend the L-fractional and the ordinary derivatives, using probability tools. A Mittag-Leffler-type function is introduced to solve linear problems, and nonlinear equations are addressed with power series, illustrating the methods for the SIR epidemic model. The inverse operator is constructed, and a fundamental theorem of calculus and an existence-and-uniqueness result for differintegral equations are proved. A conjecture on deconvolution is raised, that would permit completing the proposed theory. 
\\
\textit{AMS Classification 2020: 34A08; 34A25; 60E05; 33E12}

\section{Introduction} \label{sec_intro}

Fractional calculus is concerned with the investigation of linear operators that extend the ordinary derivative by including some sort of memory effect, for example, a continuous delay through integration. These operators are called fractional derivatives, because they normally depend on a real positive parameter, the fractional order or index, that gives standard derivatives at integer positive values. If the fractional order is not explicitly written, then one has a memory operator. Given these types of operators, new differential or differintegral equations can be defined, that include past history in the system to exhibit non-local properties. Their study, compared to ordinary differential equations, requires the development of new results on existence and uniqueness, explicit solutions (power series, Laplace transforms), and numerical methods. The reader is referred to the monographs \cite{podl_llibre,kilbas,diethelm_llibre,abbas,yong,ascione}, the review papers \cite{tenreiro,ortigueira,teodoro,dieth2}, and the research articles \cite{logistic_nieto,ovidio,compart_meu,kexue,garrap1}.

The most famous fractional derivatives are Riemann-Liouville and Caputo. Although the latter was proposed decades ago in viscoelasticity theory~\cite{caputo}, they are of use in current research, both pure and applied~\cite{lapl2,nou4,nou5}. In the definition of the operators, the past delay is incorporated by integrating with respect to a singular kernel (i.e., a function with infinite value at the endpoint of the interval). Related to the type of kernel, efforts are also being devoted to non-singular integrators, for example, bounded or continuous~\cite{singul1,singul2}. However, this kind of fractional operators has certain disadvantages, as proved in some works~\cite{dieth}.

In this paper, we consider and modify these operators proposed in the literature. We use the generic notation $Dx(t)$ here, for the operator $D$ and its evaluation at functions $x$, where $t$ is the independent variable (the time). After discussing some disadvantages of classical forms of $D$, we investigate whether dividing by $Dt$ improves the properties of the operator:
\begin{equation} \tilde{D}x(t):=\frac{Dx(t)}{Dt}. \label{tildeG_gene} \end{equation}
We also address the features of the associated fractional differential equation,
\begin{equation} \tilde{D}x(t)=f(t,x(t)). \label{tildeG_gene22} \end{equation}
In general, with~\eqref{tildeG_gene} and~\eqref{tildeG_gene22}, we will see that the values of $x'(0)$ and $\tilde{D}x(0)$ are more consistent, and that the vector field $f$ has units time$^{-1}$. For example, if $D$ contains a non-singular kernel, we will see that the units of $Dx(t)$ are time$^0$, which entails serious drawbacks, but the division by $Dt$ gives rise to a consistent rate time$^{-1}$. When $D$ has a singular kernel, for instance of Caputo type, we will see that the division by $Dt$ changes the units time$^{-\alpha}$ to time$^{-1}$, rendering alternative properties and operators. Even though the normalization with $Dt$ may seem simple, it requires the building of a new theory on fractional calculus, especially with regard to the search of explicit and closed-form solutions. In addition, it provides insight on how to design a general probabilistic definition of operator with memory that extends the ordinary derivative.

We focus on the one hand on the Caputo operator $D={}^C\! D^\alpha$, for fractional order $\alpha\in (0,1)$, due to its accepted used in the literature and its good properties when working with initial states, power series, and Laplace transforms in differential equations. It satisfies ${}^C\! D^{0^+} x(t)=x(t)-x(0)$ and, for fractional differential equations, $|x'(0)|=\infty$ and units time$^{-\alpha}$. The infinite rate of change of the dynamics at $t=0$ and the fractional units may be a problem when dealing with Caputo models, hence the need of rescaling. The normalization gives rise to the L-fractional derivative, ${}^L\! D^\alpha$, suggested a few years ago in the context of geometry and mechanics~\cite{lazo_linear,lazo_linear2} and recently studied from the mathematical-analysis point of view~\cite{arx_jo,aml_jo}. We aim at generalizing the L-fractional operator, from its probabilistic interpretation as the expectation of a certain beta-distributed delay. Albeit not the subject of our contribution, the Riemann-Liouville derivative can also be rescaled, by employing the $\Lambda$-fractional operator~\cite{lambda1}. On the other hand, besides Caputo operators, we also devote work to operators with bounded kernels, also named non-singular. The normalization of operators with bounded kernels may circumvent some of the (right) critiques to their use, such as the inconsistency at zero or the lack of inverse integral operator~\cite{dieth}. We examine the usual exponential and Mittag-Leffler kernels.

The organization of the paper is the following. In Section~\ref{subsec1_norm}, we revisit the mathematical treatment on L-fractional calculus, from~\cite{arx_jo}. In Section~\ref{subsec2_norm}, we develop novel theory on fractional operators with bounded kernels. We study how the rescaling improves their properties, for an appropriate pure and applied use. In Section~\ref{new_sec_memo}, we build a theory on new operators with memory effects, defined with probability objects. These operators extend the L-fractional and the ordinary derivatives, and they are fractional in some examples. An alternative Mittag-Leffler function is introduced. Finally, Section~\ref{sec_conc} gives limitations of the paper with open problems for the future.

Concerning notation, we denote by $\mathrm{L}^p$ the Lebesgue spaces and by $\mathcal{C}^p$ the set of functions with continuous derivatives up to order $p$. The set of continuous functions is $\mathcal{C}$. All integrals are considered in the Lebesgue sense. The norms of functions and operators use $\|\cdot\|_\ast$, where $\ast$ indicates the space. The absolute value of real numbers ($\mathbb{R}$), the modulus of complex numbers ($\mathbb{C}$), and the norms of real and complex vectors and matrices are simply expressed with $|\cdot|$. Given a probability space, the expectation is written as $\mathbb{E}[\cdot]$. When the expectation is performed with respect to a certain random variable $U$, we use $\mathbb{E}_U[\cdot]$. The essential-supremum norm, in different spaces, is $\|\cdot\|_\infty$. The composition is denoted with $\circ$. The symbol $\sim$ captures both asymptotic behavior of functions and distribution of random variables.

\section{The L-fractional derivative} \label{subsec1_norm}

This section reviews concepts from~\cite{arx_jo}, which builds a theory on L-fractional operators and linear fractional differential equations. There are several open problems there, which might of interest for readers. For nonlinear equations, one may consult~\cite{aml_jo}.

\subsection{Caputo definition}

The Caputo fractional derivative is
\begin{equation} {}^C\! D^\alpha x(t)={}^{RL}\! J^{1-\alpha} x'(t)=\frac{1}{\Gamma(1-\alpha)}\int_0^t \frac{x'(\tau)}{(t-\tau)^\alpha}\mathrm{d}\tau,
\label{derC} \end{equation}
where
\begin{equation}
 {}^{RL}\! J^{1-\alpha} x(t)=\frac{1}{\Gamma(1-\alpha)}\int_0^t (t-\tau)^{-\alpha}x(\tau)\mathrm{d}\tau
 \label{rl_intg}
\end{equation}
is the Riemann-Liouville integral, $\alpha\in (0,1)$ is the fractional order, and
\[ \Gamma(z)=\int_0^\infty \tau^{z-1}\mathrm{e}^{-\tau}\mathrm{d}\tau \]
is the gamma function (which extends the factorial). The operator~\eqref{rl_intg} is defined for Lebesgue integrable functions $x:[0,T]\rightarrow\mathbb{C}$, i.e., $x\in\mathrm{L}^1[0,T]$; the new function ${}^{RL}\! J^{1-\alpha} x$ exists almost everywhere and belongs to $\mathrm{L}^1[0,T]$, by the properties of the convolution. The Caputo derivative~\eqref{derC} takes absolutely continuous functions $x:[0,T]\rightarrow\mathbb{C}$. Recall that $x$ is said to be absolutely continuous if it is continuous, its derivative $x'$ exists almost everywhere and belongs to $\mathrm{L}^1[0,T]$, and Barrow's rule is satisfied in a Lebesgue sense, $x(t)-x(0)=\int_0^t x'(s)\mathrm{d}s$.

\subsection{L-fractional definition}

The L-fractional derivative is defined as the normalization of the Caputo operator,
\begin{equation} {}^L\! D^\alpha x(t)=\frac{{}^C\! D^\alpha x(t)}{{}^C\! D^\alpha t}=\frac{\Gamma(2-\alpha)}{t^{1-\alpha}}\,{}^C\! D^\alpha x(t). \label{derL} \end{equation}
For absolutely continuous functions $x:[0,T]\rightarrow\mathbb{C}$, the function ${}^L\! D^\alpha x$ exists almost everywhere and is in $\mathrm{L}^1[0,T]$. For $\alpha\in [0,1]$, it interpolates between 
\[ \frac{x(t)-x(0)}{t}=\frac{1}{t}\int_0^t x'(s)\mathrm{d}s, \]
which is the mean value of the velocity on $[0,T]$, and the ordinary derivative $x'(t)$ if $x\in\mathcal{C}^1[0,T]$.

When $x\in\mathcal{C}^2[0,T]$, the operator can be rewritten as
\begin{equation} {}^L\! D^\alpha x(t)=x'(0)+\frac{1}{t^{1-\alpha}}\int_0^t (t-\tau)^{1-\alpha}x''(\tau)\mathrm{d}\tau, \label{ja_ve_tr} \end{equation}
pointwise on $(0,T]$. The kernel $(t-\tau)^{1-\alpha}$ is non-singular, although the denominator $t^{1-\alpha}$ controls the value of ${}^L\! D^\alpha x(0+)$ to avoid inconsistencies: when $x\in\mathcal{C}^3[0,T]$, for example, we have
\begin{equation} {}^L\! D^\alpha x(0):=\lim_{t\rightarrow 0^+} {}^L\! D^\alpha x(t)=x'(0). \label{ja_ve_tr2} \end{equation}
In that case, ${}^L\! D^\alpha x\in\mathcal{C}[0,T]$. 

Considering~\eqref{derL}, an L-fractional differential equation is
\begin{equation} {}^L\! D^\alpha x(t)=f(t,x(t)), \label{LEDO} \end{equation}
for $t\in (0,T]$, with an initial condition or state $x(0)=x_0$, where $f:[0,T]\times \Omega\subseteq [0,T]\times\mathbb{R}^d\rightarrow\mathbb{R}^d$, or $f:[0,T]\times \Omega\subseteq [0,T]\times\mathbb{C}^d\rightarrow\mathbb{C}^d$, is a continuous function such that $x_0\in\Omega$. The equation~\eqref{LEDO} can be understood almost everywhere or everywhere on $[0,T]$, depending on whether $x$ is smooth and~\eqref{ja_ve_tr} and~\eqref{ja_ve_tr2} are of use. 

In contrast to Caputo differential equations,~\eqref{LEDO} may present smooth solutions, such as power series
\begin{equation} \sum_{n=0}^\infty x_n t^n. \label{ps_genessdz} \end{equation}
The powers are ordinary, $t^n$, instead of the fractional powers in the Caputo setting, $t^{\alpha n}$. This is related to the fact that the units of measurement in~\eqref{LEDO} are time$^{-1}$, instead of time$^{-\alpha}$. For smooth solutions, note that the property $x'(0)\in (-\infty,\infty)$ is obtained, by~\eqref{ja_ve_tr2}. Besides, the result~\eqref{ja_ve_tr2} also tells us that~\eqref{LEDO} is, locally around $t=0$, very similar to $x'(t)=f(t,x(t))$, so the change of the dynamics with $\alpha$ is smoother than in Caputo equations~\cite{aml_jo}.

There is a Mittag-Leffler-type function in this context, alternative to the classical one~\cite{mainardi}:
\begin{equation} 
\begin{split}
\mathcal{E}_{\alpha}(s)= {} & \sum_{n=0}^\infty \frac{s^n}{\Gamma(2-\alpha)^n \prod_{j=1}^n \frac{\Gamma(j+1)}{\Gamma(j+1-\alpha)}} \\
= {} & \sum_{n=0}^\infty \frac{s^n}{\Gamma(2-\alpha)^n\Gamma(1+\alpha)^n \prod_{j=1}^n \binom{j}{j-\alpha}}, 
\end{split}
\label{mlf2} 
\end{equation}
for $s\in\mathbb{C}$. The solution of
\begin{equation} {}^L\! D^\alpha x=\lambda x \label{simplest_L} \end{equation}
where $\lambda\in\mathbb{C}$ and $t\geq0$, is
\begin{equation} x(t)=\mathcal{E}_{\alpha}(\lambda t)x_0. \label{xeat} \end{equation}
Expressions~\eqref{mlf2}, \eqref{simplest_L} and~\eqref{xeat} can be defined for matrix arguments $\lambda=A\in\mathbb{C}^{d\times d}$ as well, $d\geq1$.

The integral operator associated to ${}^L\! D^\alpha$ is
\begin{equation}
 {}^L\! J^\alpha x(t)= \frac{1}{\Gamma(\alpha)\Gamma(2-\alpha)}\int_0^t (t-s)^{\alpha-1}s^{1-\alpha}x(s)\mathrm{d}s.
 \label{convolL}
\end{equation}
If $x\in\mathrm{L}^1[0,T]$, then ${}^L\! J^\alpha x\in\mathrm{L}^1[0,T]$. If $x$ is continuous on $[0,T]$, then ${}^L\! J^\alpha x$ is well defined everywhere on $[0,T]$. In fact, ${}^L\! J^\alpha:\mathcal{C}[0,T]\rightarrow\mathcal{C}[0,T]$ is a continuous operator, with norm
\[ \|{}^L\! J^\alpha\|_\infty \leq T. \]

\subsection{Connection with probability theory}

The link of L-fractional calculus with probability theory is the following: given~\eqref{derL} and~\eqref{convolL}, we have
\begin{equation} {}^L\! J^\alpha y(t)=t \mathbb{E}[y(tV)] \label{proba_L} \end{equation}
and
\begin{equation} {}^L\! D^\alpha y(t)=\mathbb{E}[y'(tW)], \label{memo_L} \end{equation}
where $V$ is a random variable with distribution $\mathrm{Beta}(2-\alpha,\alpha)$, $W$ is a random variable with distribution $\mathrm{Beta}(1,1-\alpha)$, and $\mathbb{E}$ is the expectation operator. 

In this paper, we will investigate fractional operators from the key property~\eqref{memo_L}.

\section{Fractional operators with bounded kernels} \label{subsec2_norm}

This part presents novel theory on fractional operators with non-singular kernels, based on rescaling. Our case studies are exponential and Mittag-Leffler integrators.

\subsection{Exponential kernel}

The Caputo-Fabrizio operator~\cite{cf} is defined as
\begin{equation} {}^{CF}\! D^\alpha x(t)=\frac{1}{1-\alpha}\int_0^t \mathrm{e}^{-\frac{\alpha}{1-\alpha}(t-s)}x'(s)\mathrm{d}s, 
\label{defi1_cf} 
\end{equation}
often in the context of continuously differentiable functions $x:[0,T]\rightarrow\mathbb{C}$. Notice that the kernel 
\begin{equation} \mathcal{K}(t-s)=\frac{1}{1-\alpha}\mathrm{e}^{-\frac{\alpha}{1-\alpha}(t-s)} \label{exp_kernel_cf}
\end{equation}
is non-singular, because it does not present any singularity (vanishing denominator, for example) when $t=s$. The corresponding integral operator is
\begin{equation}
 {}^{CF}\! J^\alpha x(t)=(1-\alpha)(x(t)-x(0))+\alpha\int_0^t x(s)\mathrm{d}s,
 \label{init1_cff}
\end{equation}
for continuous functions $x:[0,T]\rightarrow\mathbb{C}$. It is a convex combination between the discrete change $x(t)-x(0)$ and the continuous change $\int_0^t x(s)\mathrm{d}s$.

Although the use of a non-singular kernel may seem appealing, there are many drawbacks associated to the Caputo-Fabrizio operator, which are common in fractional calculus with bounded or continuous kernels~\cite{dieth}. First, it can be proved~\cite{losada,losada2} that
\begin{equation} {}^{CF}\! J^\alpha\circ {}^{CF}\! D^\alpha x(t)=x(t)-x(0) \label{first_one_cf} \end{equation}
and
\begin{equation}  {}^{CF}\! D^\alpha\circ {}^{CF}\! J^\alpha x(t)=x(t)-\mathrm{e}^{-\frac{\alpha}{1-\alpha}t}x(0). \label{second_one_cf} \end{equation}
Note that~\eqref{first_one_cf} mimics Barrow's rule, but~\eqref{second_one_cf} does not correspond to the fundamental theorem of calculus: the derivative of the integral is not the identity operator. This fact is related with the null space of~\eqref{init1_cff}, given by 
\[ \langle \mathrm{e}^{-\frac{\alpha}{1-\alpha}t} \rangle, \]
where $\langle \cdot \rangle$ denotes the linear span. Second, when proposing a model of the form
\begin{equation}
 {}^{CF}\! D^\alpha x(t)=f(t,x(t)),
 \label{edo_cf_co}
\end{equation}
there is a clear contradiction at $t=0$, since
\[ {}^{CF}\! D^\alpha x(0)=0 \]
is always satisfied. Hence the only possibility is to work with the integral problem 
\begin{equation} x(t)=x_0+{}^{CF}\! J^\alpha f(t,x(t)), \label{veig_gran}
\end{equation}
which is not entirely equivalent to~\eqref{edo_cf_co}, due to~\eqref{second_one_cf}. And third, the units of~\eqref{defi1_cf} are time$^{0}$, i.e., it has no time units. Indeed, the units of $x'(s)$ cancel out with those of $\mathrm{d}s$, while the exponential kernel~\eqref{exp_kernel_cf} is non-singular and does not behave as time$^{-\beta}$ for any index $\beta>0$: by Taylor expansion,
\[ \mathrm{e}^{-\frac{\alpha}{1-\alpha}(t-s)}=1-\frac{\alpha}{1-\alpha}(t-s)+\frac12 \left(\frac{\alpha}{1-\alpha}\right)^2(t-s)^2+\ldots. \]

These disadvantages can be resolved by working with the normalized version of~\eqref{defi1_cf}. Simple computations yield
\begin{equation} {}^{CF}\! D^\alpha t=\frac{1}{\alpha}\left(1-\mathrm{e}^{-\frac{\alpha}{1-\alpha}t}\right)\sim \frac{1}{1-\alpha}t, \label{Dt_cf}
\end{equation}
where $\sim$ denotes here that the two functions are asymptotically equivalent when $t\rightarrow0^+$, which is the problematic point. Considering~\eqref{Dt_cf}, we define
\[ c_\alpha(t)=\frac{1}{\alpha}\left(1-\mathrm{e}^{-\frac{\alpha}{1-\alpha}t}\right) \]
and
\begin{equation}
 {}^{NCF}\! D^\alpha x(t)=\frac{1}{(1-\alpha)c_\alpha(t)}\int_0^t \mathrm{e}^{-\frac{\alpha}{1-\alpha}(t-s)}x'(s)\mathrm{d}s. 
\label{defi222_cf} 
\end{equation}
Notice that convergence towards $x'(t)$ when $\alpha\rightarrow 1^-$ holds for~\eqref{defi222_cf}. Also,
\begin{equation} {}^{NCF}\! D^\alpha x(0^+)=\frac{x(t)-x(0)}{t}, \label{molt_promk} \end{equation}
which is the mean value
\[ \frac{1}{t}\int_0^t x'(s)\mathrm{d}s. \]
That is,~\eqref{defi222_cf} interpolates between the mean velocity on $[0,t]$ and at $t$.

It is clear that the units of~\eqref{defi222_cf} are time$^{-1}$, by the division by $c_\alpha(t)$. Therefore, the new operator represents some sort of rate of change, in contrast to the standard one~\eqref{defi1_cf}. Furthermore, now~\eqref{defi222_cf} is a convolution with respect to a probability density function on $[0,t]$, for each $t>0$, so that the weight function seems to be more appropriate.

Definition~\eqref{defi222_cf} is well posed for $t=0$:
\begin{equation}
 {}^{NCF}\! D^\alpha x(0)=\lim_{t\rightarrow0^+} \mathrm{e}^{-\frac{\alpha}{1-\alpha}t}\cdot \frac{1}{t}\int_0^t \mathrm{e}^{\frac{\alpha}{1-\alpha}s}x'(s)\mathrm{d}s=x'(0). 
\label{defi222_cf_Limit} 
\end{equation}
Even for $\alpha=0$, it is true that ${}^{NCF}\! D^\alpha x(0)=x'(0)$, by~\eqref{molt_promk}. In~\eqref{defi222_cf_Limit}, we used $(1-\alpha)c_\alpha(t)\sim t$ and decomposed the exponential function to readily apply the fundamental theorem of calculus, but for a general continuous kernel, one can proceed with Leibniz rule of differentiation for integrals or integration by parts and arrive at the same result. Thus, the new fractional operator has an appropriate value at $t=0$, the same as the integer-order one, and one can work with fractional differential equations of the form
\begin{equation} {}^{NCF}\! D^\alpha x(t)=f(t,x(t)), \label{prem_ncf} \end{equation}
as opposed to~\eqref{edo_cf_co}. The associated integral operator (see~\eqref{init1_cff}),
\begin{equation} {}^{NCF}\! J^\alpha x(t)=(1-\alpha)c_\alpha(t)x(t)+\alpha \int_0^t c_\alpha(s) x(s)\mathrm{d}s, \label{assoc_cf_insi} \end{equation}
does satisfy the fundamental theorem of calculus, in contrast to~\eqref{second_one_cf}: for a continuously differentiable function $x$,
\begin{align*}
 {}^{NCF}\! J^\alpha\circ {}^{NCF}\! D^\alpha x(t)= {} & {}^{NCF}\! J^\alpha\left[\frac{1}{c_\alpha}\cdot {}^{CF}\! D^\alpha x\right](t) \\
= {} & {}^{CF}\! J^\alpha \circ {}^{CF}\! D^\alpha x \\
= {} & x(t)-x(0) 
\end{align*}
and
\begin{equation} 
\begin{split}
{}^{NCF}\! D^\alpha\circ {}^{NCF}\! J^\alpha x(t)= {} & {}^{NCF}\! D^\alpha \circ {}^{CF}\! J^\alpha [c_\alpha x](t) \\
= {} & \frac{1}{c_\alpha(t)}\cdot{}^{CF}\! D^\alpha \circ {}^{CF}\! J^\alpha [c_\alpha x](t) \\
= {} & \frac{1}{c_\alpha(t)} c_\alpha(t)x(t) \\
= {} & x(t), 
\end{split}
\label{comp_finij} 
\end{equation}
since $c_\alpha(t)x(t)$ is $0$ at $t=0$. We employed both~\eqref{first_one_cf} and~\eqref{second_one_cf}. Observe that $t=0$ is not a problem for~\eqref{comp_finij}, by~\eqref{defi222_cf_Limit} and~\eqref{assoc_cf_insi}:
\[ {}^{NCF}\! D^\alpha\left[{}^{NCF}\! J^\alpha x\right](0)=\left({}^{NCF}\! J^\alpha x\right)'(0)=x(0). \]
Equation~\eqref{prem_ncf} is now completely equivalent to
\begin{equation} x(t)=x_0 +{}^{NCF}\! J^\alpha f(t,x(t)), \label{dsfdjkcF} \end{equation}
in the set of $\mathcal{C}^1[0,T]$ functions.

We do not enter into other possible issues regarding Caputo-Fabrizio models. For example,~\eqref{veig_gran} is equivalent to an ordinary differential equation, by differentiating. The same fact occurs for the alternative model~\eqref{prem_ncf}, which gives
\begin{equation}
\begin{split} x'(t)= {} & (1-\alpha)c_\alpha'(t)f(t,x(t)) \\
{} & +(1-\alpha)c_\alpha(t)[f_t(t,x(t))+x'(t)f_x(t,x(t))]+\alpha c_\alpha(t)f(t,x(t)),
\end{split}
\label{edo_Cff}
\end{equation}
by~\eqref{assoc_cf_insi} and~\eqref{dsfdjkcF}, so the non-local behavior of~\eqref{prem_ncf} is debatable. As~\eqref{edo_Cff} is well understood, fractional models with exponential kernel do not seem to give rise to a big new theory, except concrete mathematical results. When we convolve the derivative $x'$ with an exponential in~\eqref{defi1_cf}, the exponential function decomposes as a product and separates the variables $t$ (outside the integral) and $s$ (inside the integral), and this gives rise to an ordinary differential equation at the end. Since the weight function in~\eqref{defi222_cf} is related to the exponential probability distribution, which exhibits the memoryless property, it is intuitive that the non-local behavior of~\eqref{prem_ncf} is lost.

\begin{Example} \normalfont
The model 
\begin{equation} {}^{CF}\! D^\alpha x(t)=-\lambda x(t), 
\label{cfex}
\end{equation}
where $t>0$ and $\lambda>0$, extends the exponential decay equation to a fractional setting, for $\alpha\in (0,1]$. First, observe that, if $x(0)=x_0\neq0$, then~\eqref{cfex} is not well-defined, because ${}^{CF}\! D^\alpha x(0)=0\neq -\lambda x_0$. Besides, as already proved, the units of ${}^{CF}\! D^\alpha x(t)$ are time$^0$, hence there is no valid power for $\lambda$; it does not represent any rate. Finally, the differintegral equation~\eqref{cfex} is not equivalent to the integral equation $x(t)=x_0+{}^{CF}\! J^\alpha[-\lambda x(t)]$, by~\eqref{second_one_cf}. The alternative model
\begin{equation} {}^{NCF}\! D^\alpha x(t)=-\lambda x(t)
\label{cfex2}
\end{equation}
fixes all these issues. It is consistent at $t=0$, with ${}^{NCF}\! D^\alpha x(0)=x'(0)=-\lambda x_0$, the units of ${}^{NCF}\! D^\alpha x(t)$ and $\lambda$ are time$^{-1}$, thus representing true rates, and finally, the differintegral equation~\eqref{cfex2} is equivalent to the integral equation $x(t)=x_0+{}^{NCF}\! J^\alpha[-\lambda x(t)]$, because the fundamental theorem of calculus holds. The argumentation for this simple example~\eqref{cfex2} works for any other model defined via a non-singular kernel. During the time the present paper has been in the preprint server ArXiv and under review, specific models based on my proposed normalized operators are being studied~\cite{balen_nou}.
\end{Example}

\subsection{Mittag-Leffler kernel}

Other possible operators with bounded kernels are based on Mittag-Leffler kernels, instead of exponential ones~\cite{atang,prab_area}. For example,
\begin{equation} {}^{AB}\! D^\alpha x(t)=\frac{1}{1-\alpha}\int_0^t E_\alpha\left(-\frac{\alpha}{1-\alpha}(t-s)^\alpha\right)x'(s)\mathrm{d}s. \label{AB_ker} \end{equation}
Here one often considers absolutely continuous functions $x:[0,T]\rightarrow\mathbb{C}$. The kernel
\[ \mathcal{K}(t-s)=\frac{1}{1-\alpha}E_\alpha\left(-\frac{\alpha}{1-\alpha}(t-s)^\alpha\right) \]
is non-singular. Observe that, again,
\[ {}^{AB}\! D^\alpha x(0)=0, \]
which is a severe restriction in differential equations (exactly the same issue as for Caputo-Fabrizio models)~\cite{dieth}. The units of~\eqref{AB_ker} are also time$^{0}$ (the use of the exponent $\alpha$ does not change this fact). Since
\[ {}^{AB}\! D^\alpha t^{\gamma}=\frac{1}{1-\alpha}\Gamma(1+\gamma)t^\gamma E_{\alpha,\gamma+1}\left(-\frac{\alpha}{1-\alpha}t^\alpha\right),
\]
where $\gamma>0$, the rescaled operator associated to~\eqref{AB_ker} is
\begin{equation} {}^{NAB}\! D^\alpha x(t)=\frac{1}{(1-\alpha)k_\alpha(t)}\int_0^t E_\alpha\left(-\frac{\alpha}{1-\alpha}(t-s)^\alpha\right)x'(s)\mathrm{d}s, \label{AB_kerr2} \end{equation}
being
\[ k_\alpha(t)=\frac{1}{1-\alpha}tE_{\alpha,2}\left(-\frac{\alpha}{1-\alpha}t^\alpha\right)\sim \frac{1}{1-\alpha}t. \]
The new fractional operator~\eqref{AB_kerr2} satisfies
\[ {}^{NAB}\! D^\alpha x(0)=x'(0) \]
when calculating the limit as $t\rightarrow0^+$, if $x'$ is continuous, by Leibniz rule of differentiation for integrals:
\begin{align*}
\lim_{t\rightarrow0^+} {}^{NAB}\! D^\alpha x(t)= {} & \lim_{t\rightarrow0^+}\frac{1}{t}\int_0^t E_\alpha\left(-\frac{\alpha}{1-\alpha}(t-s)^\alpha\right)x'(s)\mathrm{d}s \\
= {} & \left.\frac{\mathrm{d}}{\mathrm{d}t}\right|_{t=0^+} \int_0^t E_\alpha\left(-\frac{\alpha}{1-\alpha}(t-s)^\alpha\right)x'(s)\mathrm{d}s \\
= {} & x'(0) + \frac{\alpha^2}{1-\alpha}\lim_{t\rightarrow0^+} \int_0^t (t-s)^{\alpha-1} E_\alpha'\left(-\frac{\alpha}{1-\alpha}(t-s)^\alpha\right)x'(s)\mathrm{d}s \\
= {} & x'(0).
\end{align*}
Hence there are no issues at $t=0$, compared with~\eqref{AB_ker}. The units of~\eqref{AB_kerr2} are time$^{-1}$. Now~\eqref{AB_kerr2} is a convolution with respect to a probability density function on $[0,t]$, for each $t>0$, so that the weight function is likely more adequate.

The integral operator of~\eqref{AB_ker} is
\begin{equation}  {}^{AB}\! J^\alpha x(t)=(1-\alpha)(x(t)-x(0))+\alpha\cdot {}^{RL}\! J^\alpha x(t),
\label{AB_integ_curt}
\end{equation}
which is a convex combination of the discrete change $x(t)-x(0)$ and the continuous delayed change ${}^{RL}\! J^\alpha x(t)$. From~\eqref{AB_integ_curt}, one deduces that
\begin{equation}  {}^{NAB}\! J^\alpha x(t)=(1-\alpha)k_\alpha(t)x(t)+\alpha\cdot {}^{RL}\! J^\alpha [k_\alpha x](t)
\label{pelletssi}
\end{equation}
is the integral counterpart of~\eqref{AB_kerr2}. It is well known that
\[ {}^{AB}\! J^\alpha\circ {}^{AB}\! D^\alpha x(t)=x(t)-x(0) \]
and
\[  {}^{AB}\! D^\alpha\circ {}^{AB}\! J^\alpha x(t)=x(t)-E_\alpha\left(-\frac{\alpha}{1-\alpha}t^\alpha\right)x(0)\neq x(t), \]
but now
\begin{align*}
 {}^{NAB}\! J^\alpha\circ {}^{NAB}\! D^\alpha x(t)= {} & {}^{NAB}\! J^\alpha\left[\frac{1}{k_\alpha(t)}\cdot {}^{AB}\! D^\alpha x\right](t) \\
= {} & {}^{AB}\! J^\alpha \circ {}^{AB}\! D^\alpha x \\
= {} & x(t)-x(0) 
\end{align*}
and
\begin{align*}
 {}^{NAB}\! D^\alpha\circ {}^{NAB}\! J^\alpha x(t)= {} & {}^{NAB}\! D^\alpha \circ {}^{AB}\! J^\alpha [k_\alpha x](t) \\
= {} & \frac{1}{t}\cdot{}^{AB}\! D^\alpha \circ {}^{AB}\! J^\alpha [k_\alpha x](t) \\
= {} & \frac{1}{k_\alpha(t)} k_\alpha(t)x(t) \\
= {} & x(t), 
\end{align*}
because $k_\alpha(0)x(0)=0$. Therefore, the fundamental theorem of calculus for the new operator ${}^{NAB}\! D^\alpha$ is verified.

We have the equivalence between 
\[ {}^{NAB}\! D^\alpha x(t)=f(t,x(t)) \]
and
\[ x(t)=x_0+{}^{NAB}\! J^\alpha f(t,x(t)), \]
in the set of $\mathcal{C}^1[0,T]$ functions. This model is the same as a system with implicit Caputo derivatives only,
\[ {}^{C}\! D^\alpha x(t)=(1-\alpha)\cdot {}^{C}\! D^\alpha(k_\alpha(t) f(t,x(t)))+\alpha k_\alpha(t) f(t,x(t)) \]
almost everywhere, by~\eqref{pelletssi} and the fundamental theorem of calculus. This fact is natural somehow, because the Mittag-Leffler function is directly related to Caputo fractional differential equations (analogously, the exponential function in the Caputo-Fabrizio model is directly related to ordinary differential equations): it builds the solution of the basic equation ${}^{C}\! D^\alpha z(t)=\lambda^\alpha z(t)$ (analogously, the exponential function builds the solution of the basic equation $z'(t)=\lambda z(t)$).

In conclusion, the factor $(Dt)^{-1}$ seems to resolve major issues associated to fractional operators with bounded kernels. Some authors proposed the use of integration by parts, directly on~\eqref{defi1_cf} and~\eqref{AB_ker}, see~\cite{dieth,baleanuP}:
\begin{equation} {}^{CF}\! \tilde{D}^\alpha x(t)=\frac{1}{1-\alpha}\left[ x(t)-\mathrm{e}^{-\frac{\alpha}{1-\alpha}t}x(0)-\frac{\alpha}{1-\alpha}\int_0^t \mathrm{e}^{-\frac{\alpha}{1-\alpha}(t-s)}x(s)\mathrm{d}s\right], \label{novds1} \end{equation}
\begin{equation} 
\begin{split}
{}^{AB}\! \tilde{D}^\alpha x(t)= {} & \frac{1}{1-\alpha}\left[ x(t)-E_\alpha\left(-\frac{\alpha}{1-\alpha}t^\alpha\right)x(0)\right. \\
{} & \left.-\frac{\alpha}{1-\alpha}\int_0^t E_\alpha\left(-\frac{\alpha}{1-\alpha}(t-s)^\alpha\right)x(s)\mathrm{d}s\right], 
\end{split}
\label{novds2} \end{equation}
for Lebesgue integrable functions $x$. However, since there is an evaluation at $0$, continuity of $x$ at $t=0$ is needed for uniqueness of the fractional derivative. In that case, the operators~\eqref{novds1} and~\eqref{novds2} still tend to $0$ when $t\rightarrow0^+$, and the inconsistency persists. Besides, the issue of the units time$^{0}$ is present again. There is no derivative in~\eqref{novds1} and~\eqref{novds2}, which contradicts the notion of fractional \textit{derivative}; indeed, some rate of change should be involved in the formulation.

In terms of explicit or closed-form solutions, a disadvantage of the rescaling of fractional operators in differential equations is that the applicability of the Laplace-transform method is reduced, because the transform of the product or division is not amenable to computing. Nonetheless, the power-series technique may still be of application, as one can find the power series of a product easily.

\section{A new operator with memory effects using probability theory} \label{new_sec_memo}

This part of the paper generalizes normalized operators by employing probability concepts. A theory on the new operators is built.

\subsection{Definition} \label{subs_def_WW}

By generalizing the probabilistic interpretation of the L-fractional derivative, see~\eqref{memo_L}, we define the new linear operator
\begin{equation}
 \mathcal{D}x(t)=\mathbb{E}[x'(tW)],
 \label{new_fr_derii}
\end{equation}
where $t\in [0,T]$, $x:[0,T]\rightarrow\mathbb{C}$ is (at least) an absolutely continuous function, and $W$ is a random variable with the requirements:
\begin{equation}
 \mathrm{support}(W)\subseteq [0,1],
 \label{suportW}
\end{equation}
\begin{equation}
 \|W\|_\infty=1,
 \label{normW}
\end{equation}
and
\begin{equation}
 \lim_{n\rightarrow\infty} n\mathbb{E}[W^n]=\infty.
 \label{limitWW}
\end{equation}
The notation $\|\cdot\|_\infty$ is the essential supremum of the random variable. Both~\eqref{suportW} and~\eqref{normW} are related to the need of capturing the past history of $x'$ until the present, along $[0,t]$, i.e., exhibiting memory effects. Physically, the new operator~\eqref{new_fr_derii} is a weighted average of the velocity on $[0,t]$, where the weight is modeled with a probability distribution. The third condition~\eqref{limitWW} will be required later, when dealing with differential equations and series, so that the new ``exponential'' function exists on $\mathbb{R}$. Observe that~\eqref{limitWW} implies~\eqref{normW}, but we explicitly write~\eqref{normW} for the sake of clarity. Assumption~\eqref{limitWW} means that the past time near the present $t$ should have enough weight in the memory operator.

We observe that $x'$ exists almost everywhere. In definition~\eqref{new_fr_derii}, we are assuming that
\begin{equation} \mathbb{E}[|x'(tW)|]<\infty. \label{eperr_fini} \end{equation}
For example,~\eqref{eperr_fini} follows if $x'$ is essentially bounded on $[0,T]$, i.e., $\|x'\|_\infty<\infty$.

Some examples of $\mathcal{D}$ are the following:
\begin{itemize}
\item When $W=1$ is constant, then $\mathcal{D}x=x'$ is the classical derivative. Note that the three assumptions~\eqref{suportW}, \eqref{normW} and~\eqref{limitWW} hold.
\item When $W\sim \mathrm{Beta}(1,1-\alpha)$, for $\alpha\in (0,1)$, then $\mathcal{D}x={}^L\! D^\alpha x$ is the L-fractional derivative. The third condition~\eqref{limitWW} is satisfied:
\begin{align*}
\lim_{n\rightarrow\infty} n\mathbb{E}[W^n]= {} & \lim_{n\rightarrow\infty} n \frac{\Gamma(2-\alpha)\Gamma(n)}{\Gamma(n+1-\alpha)} \\
= {} & \lim_{n\rightarrow\infty} n \frac{\Gamma(2-\alpha)}{n^{1-\alpha}} \\
= {} & \infty.
\end{align*}
\item If $W\sim \mathrm{Beta}(\alpha,\beta)$, for $\alpha,\beta\in (0,1)$, then we obtain a generalization of the L-fractional derivative~\cite{arx_jo}, denoted as ${}^{GL}\! D^{\alpha,\beta}$. The condition~\eqref{limitWW} is verified as follows:
\begin{equation}
\begin{split}
\lim_{n\rightarrow\infty} n\mathbb{E}[W^n]= {} & \lim_{n\rightarrow\infty} n \frac{\Gamma(\alpha+\beta)\Gamma(\alpha+n-1)}{\Gamma(\alpha)\Gamma(\alpha+n-1+\beta)} \\
= {} & \lim_{n\rightarrow\infty} n \frac{\Gamma(\alpha+\beta)}{\Gamma(\alpha)(\alpha+n-1)^\beta} \\
= {} & \infty.
\end{split}
 \label{label_limitBet}
\end{equation}

Regarding~\eqref{eperr_fini}, simple computations for integrals yield
\begin{equation}
\begin{split}
 {}^{GL}\! D^{\alpha,\beta}x(t)= {} & \frac{\Gamma(\alpha+\beta)}{\Gamma(\alpha)\Gamma(\beta)}\int_0^1 s^{\alpha-1}(1-s)^{\beta-1} x'(ts)\mathrm{d}s \\
= {} & \frac{\Gamma(\alpha+\beta)}{\Gamma(\alpha)\Gamma(\beta)}\frac{1}{t}\int_0^t \frac{s^{\alpha-1}}{t^{\alpha-1}}\left(1-\frac{s}{t}\right)^{\beta-1}x'(s)\mathrm{d}s \\
= {} & \frac{\Gamma(\alpha+\beta)}{\Gamma(\alpha)\Gamma(\beta)}\frac{1}{t^{\alpha+\beta-1}}\int_0^t (t-s)^{\beta-1}s^{\alpha-1}x'(s)\mathrm{d}s \\
= {} & \frac{1}{t^{\alpha+\beta-1}} \left(\frac{\Gamma(\alpha+\beta)}{\Gamma(\alpha)\Gamma(\beta)} t^{\beta-1}\ast (t^{\alpha-1}x')\right)(t),
\end{split}
\label{GL_forma_funci}
\end{equation}
for $t>0$. The density function of the beta distribution and the expression for the expectation have been used. Thus, if $x$ is absolutely continuous and 
\begin{equation} t^{\alpha-1}x'\in\mathrm{L}^1[0,T], \label{que_babd} \end{equation}
then ${}^{GL}\! D^{\alpha,\beta}x$ exists almost everywhere and belongs to $\mathrm{L}^1[0,T]$ (recall that a convolution $\mathcal{K}\ast y$ is well defined almost everywhere and belongs to $\mathrm{L}^1[0,T]$ if $\mathcal{K},y\in\mathrm{L}^1[0,T]$). Note that, if $x'$ is essentially bounded on $[0,T]$, that is, $\|x'\|_\infty<\infty$, the integrability condition~\eqref{que_babd} follows. The point $t=0$ is not a problem, because
\[ {}^{GL}\! D^{\alpha,\beta}x(t)=\mathbb{E}[y'(0\cdot W)]=y'(0) \]
is well defined. Note that~\eqref{GL_forma_funci} clearly extends the integral expression for the L-fractional operator, thus modifying the Caputo operator as well.

Based on the L-fractional derivative, when $\alpha\rightarrow 1^-$, $\beta\rightarrow0^+$ and $x$ is continuously differentiable on $[0,T]$, the ordinary derivative operator is retrieved. If $\alpha\rightarrow 1^-$ and $\beta\rightarrow1^-$, then the mean value
\[ \frac{1}{t}\int_0^t x'(s)\mathrm{d}s \]
is obtained.

\item If $W\sim\mathrm{Uniform}(0,1)$, i.e., $\mathrm{Beta}(1,1)$, then~\eqref{limitWW} does not hold: by~\eqref{label_limitBet},
\[ \lim_{n\rightarrow\infty} n\mathbb{E}[W^n]=1. \]
In consequence, that simple probability distribution is not valid. This is an interesting example, because the uniform distribution should be a good option a priori, as it maximizes the Shannon entropy (the ignorance) on $[0,t]$ when multiplied by $t$ if there is no information available on the weight~\cite{sema}.

\item Important cases related with the gamma and exponential distributions will be given in Example~\ref{ex_m_isk} and Example~\ref{ex_m_isk2}, in the context of inverse operators and the fundamental theorem of calculus. New operators with memory will be obtained.
\end{itemize}

Given~\eqref{new_fr_derii}, one can define the concept of differential equation as
\begin{equation}
 \mathcal{D}x(t)=f(t,x(t)).
 \label{ode_DDDD}
\end{equation}
In general, we work in dimension $d\geq1$, with $x:[0,T]\rightarrow\mathbb{C}^d$. Equation~\eqref{ode_DDDD} has an initial condition or state $x(0)=x_0\in\mathbb{C}^d$, where $f:[0,T]\times \Omega\subseteq [0,T]\times\mathbb{R}^d\rightarrow\mathbb{R}^d$, or $f:[0,T]\times \Omega\subseteq [0,T]\times\mathbb{C}^d\rightarrow\mathbb{C}^d$, is a continuous function such that $x_0\in\Omega$. The units in~\eqref{ode_DDDD} are time$^{-1}$.

Compared to an ordinary differential equation $x'(t)=f(t,x(t))$, now we are relating the weighted average of the velocity on $[0,t]$ and the position at $t$. In differential form,
\[ \mathbb{E}[x(tW+\mathrm{d}t)]=\mathbb{E}[x(tW)]+f(t,x(t))\mathrm{d}t. \]
The connection between the past $[0,t]$ and the future $(t,t+\mathrm{d}t]$ is stronger than in the ordinary case, similar to the difference between non-Markovian and Markovian processes. If $x$ is changed at $t$, then this has an effect on the whole history of $x$, not just on $(t,t+\mathrm{d}t]$. In this work, we deal with~\eqref{ode_DDDD} abstractly.

\subsection{Power series and Mittag-Leffler-type function} \label{subs_pankds}

We study the differentiation rules of power series.

\begin{Proposition} \label{propisa}
Consider the operator $\mathcal{D}$ from Section~\ref{subs_def_WW}. If
\begin{equation} x(t)=\sum_{n=0}^\infty x_n t^n \label{sumsid} \end{equation}
is a convergent power series on $[0,T]$, where $x_n\in\mathbb{C}^d$, then
\begin{equation} \mathcal{D}x(t)=\sum_{n=0}^\infty x_n \mathcal{D}t^n=\sum_{n=0}^\infty x_{n+1}(n+1)\mathbb{E}[W^n]t^n, \label{eachIjfis} \end{equation}
pointwise on $[0,T]$.
\end{Proposition}
\begin{proof}
Since $x$ is bounded on $[0,T]$, condition~\eqref{eperr_fini} is met and $\mathcal{D}x$ exists pointwise on $[0,T]$. Let 
\[ x_N(t)=\sum_{n=0}^N x_n t^n \]
be the truncated sum of~\eqref{sumsid}. It is well-known that
\[ \lim_{N\rightarrow\infty} \| x_N-x\|_\infty=0 \]
and
\[ \lim_{N\rightarrow\infty} \| x_N'-x'\|_\infty=0 \]
on $[0,T]$. Also, by linearity,
\[ \mathcal{D} x_N(t)=\sum_{n=0}^N x_n \mathcal{D} t^n. \]
Then, from the definition~\eqref{new_fr_derii} for $\mathcal{D}$, we derive
\[ \|\mathcal{D} x-\mathcal{D} x_N\|_\infty \leq \| x'-x_N'\|_\infty\stackrel{n\rightarrow\infty}{\longrightarrow}0. \]
Thus, the first equality of~\eqref{eachIjfis} is proved. For the second equality in~\eqref{eachIjfis}, just compute
\[ \mathcal{D}t^n=nt^{n-1}\mathbb{E}[W^{n-1}], \]
for $n\geq1$.
\end{proof}

In the context of Section~\ref{subs_def_WW}, we define the new Mittag-Leffler-type function
\begin{equation} \mathcal{E}(s)=\sum_{n=0}^\infty \frac{s^n}{n! \prod_{j=1}^{n-1} \mathbb{E}[W^j]}, \label{mlf_kon}
\end{equation}
for $s\in\mathbb{C}$. The empty product is interpreted, as usual, as $1$. Of course, when $W\sim\mathrm{Beta}(1,1-\alpha)$, we have the L-fractional derivative and the Mittag-Leffler-type function~\eqref{mlf2} is retrieved, after computing the moments of that beta distribution. When $W=1$, the usual exponential function is obtained. Observe that~\eqref{mlf_kon} is convergent on $\mathbb{C}$, by the ratio test and condition~\eqref{limitWW}:
\[ \lim_{n\rightarrow\infty} \frac{n! \prod_{j=1}^{n-1} \mathbb{E}[W^j]}{(n+1)! \prod_{j=1}^{n} \mathbb{E}[W^j]}=\lim_{n\rightarrow\infty} \frac{1}{(n+1)\mathbb{E}[W^n]}=0. \]
The same reasoning can be conducted for matrix arguments $s=A\in\mathbb{C}^{d\times d}$. The function~\eqref{mlf_kon} is key to solve linear models.

\begin{Proposition}
Consider the operator $\mathcal{D}$ from Section~\ref{subs_def_WW}. If $A\in\mathbb{C}^{d\times d}$ and $x_0\in\mathbb{C}^d$, then 
\[ x(t)=\mathcal{E}(tA)x_0 \]
solves
\[ \mathcal{D}x(t)=Ax(t) \]
with $x(0)=x_0$, pointwise on $[0,\infty)$.
\end{Proposition}
\begin{proof}
By Proposition~\eqref{propisa},
\begin{align*}
 \mathcal{D}x(t)= {} & \sum_{n=0}^\infty A^{n+1}x_0 \frac{(n+1)\mathbb{E}[W^n]}{(n+1)! \prod_{j=1}^{n} \mathbb{E}[W^j]}t^n \\
= {} & A \sum_{n=0}^\infty A^{n}x_0 \frac{1}{n! \prod_{j=1}^{n-1} \mathbb{E}[W^j]}t^n \\
= {} & Ax(t).
\end{align*}
\end{proof}

\subsection{Fundamental theorem of calculus} \label{subsec_conjesdz}

Let $W$ be a fixed random variable satisfying~\eqref{suportW}, \eqref{normW} and~\eqref{limitWW}. Suppose that there exists another random variable $V$, independent of $W$ and with support in $[0,1]$, such that 
\begin{equation} W V=U\sim\mathrm{Uniform}(0,1). \label{taabbs} \end{equation}
For example, when $W\sim\mathrm{Beta}(1,1-\alpha)$ in L-fractional calculus, $\alpha\in (0,1)$, then $V\sim\mathrm{Beta}(2-\alpha,\alpha)$, because
\begin{align*}
\mathbb{E}[U^n]= {} & \mathbb{E}[ W^n]\mathbb{E}[V^n] \\
= {} & \frac{\Gamma(1+n)\Gamma(1-\alpha)}{\Gamma(2+n-\alpha)}\frac{\Gamma(2-\alpha)}{\Gamma(1)\Gamma(1-\alpha)} \\
 {} & \times \frac{\Gamma(2-\alpha+n)\Gamma(\alpha)}{\Gamma(2+n)}\frac{\Gamma(2)}{\Gamma(2-\alpha)\Gamma(\alpha)} \\
= {} & \frac{1}{1+n}
\end{align*}
are the moments of a $\mathrm{Uniform}(0,1)$ distribution (which uniquely determine it).

\begin{Remark} \label{tamifns} \normalfont
In general, the structure of $V$ in~\eqref{taabbs} is not simple. Let $W\sim\mathrm{Beta}(\alpha,\beta)$ with $0<\alpha,\beta<1$. Suppose that $V\sim\mathrm{Beta}(\tilde{\alpha},\tilde{\beta})$. Equate moments
\[ \mathbb{E}[V^r]=\frac{\mathbb{E}[U^r]}{\mathbb{E}[W^r]} \]
for $r\geq1$:
\begin{equation} \frac{\Gamma(\tilde{\alpha}+r)\Gamma(\tilde{\alpha}+\tilde{\beta})}{\Gamma(\tilde{\alpha}+r+\tilde{\beta})\Gamma(\tilde{\alpha})}=\frac{\Gamma(\alpha)\Gamma(\alpha+\beta+r)}{\Gamma(\alpha+\beta)(1+r)\Gamma(\alpha+r)}. \label{endksld} \end{equation}
Asymptotically, the left-hand side of~\eqref{endksld} is of the form
\[ \frac{1}{(\tilde{\alpha}+r)^{\tilde{\beta}}}\frac{\Gamma(\tilde{\alpha}+\tilde{\beta})}{\Gamma(\tilde{\alpha})}, \]
while the right-hand side of~\eqref{endksld} is
\[ \frac{\Gamma(\alpha)}{\Gamma(\alpha+\beta)}\frac{(\alpha+r)^\beta}{1+r}. \]
In consequence,
\[ \tilde{\beta}=1-\beta\in (0,1) \]
and
\begin{equation} \frac{\Gamma(\tilde{\alpha}+1-\beta)}{\Gamma(\tilde{\alpha})}=\frac{\Gamma(\alpha)}{\Gamma(\alpha+\beta)}. \label{impooads} \end{equation}
Since $\alpha,\beta\in (0,1)$, we know that 
\begin{equation}
\Gamma(\alpha)>\Gamma(\alpha+\beta). 
\label{impooads22} 
\end{equation}
If $\tilde{\alpha}\in (0,1]$, then we would have $\Gamma(\tilde{\alpha})>\Gamma(\tilde{\alpha}+1-\beta)$, but this is impossible considering~\eqref{impooads} and~\eqref{impooads22}. Hence
\[ \tilde{\alpha}> 1. \]
On the other hand,~\eqref{endksld} becomes
\begin{equation} \frac{\Gamma(\tilde{\alpha}+r)}{\Gamma(\tilde{\alpha}+r+1-\beta)}=\frac{\Gamma(\alpha+\beta+r)}{(1+r)\Gamma(\alpha+r)}, \label{eruirskj} \end{equation}
for all $r\geq1$. Let $\mu=\tilde{\alpha}-\beta$. Observe that 
\[ (1+r)\Gamma(\alpha+r)> (\alpha+r)\Gamma(\alpha+r)=\Gamma(\alpha+r+1), \]
so, from~\eqref{eruirskj},
\[ \frac{\Gamma(\mu+r+\beta)}{\Gamma(\mu+r+1)}<\frac{\Gamma(\alpha+r+\beta)}{\Gamma(\alpha+r+1)}. \]
Since $\Gamma$ increases faster and faster as $r$ grows, we deduce that $\mu>\alpha$, that is,
\[ \tilde{\alpha}>\min\{1,\alpha+\beta\}. \]
Computations in software show that, if $\alpha=0.99$ and $\beta=0.5$, for example, then the unique root of~\eqref{eruirskj} for $r=1$ is $\tilde{\alpha}\approx 1.51284$, whereas for $r=2$ it is $\tilde{\alpha}\approx 1.51183$, which is distinct; therefore, $V$ cannot follow the beta distribution. For $\alpha=1$ and $\beta\in (0,1)$, we are in the situation of L-fractional calculus, where $V$ is indeed beta distributed. The root $\tilde{\alpha}$ in that case is unique.
\end{Remark}

An open problem regarding~\eqref{taabbs} is the following:

\begin{conjecture} \label{conj_WWW}
Let $W$ be a random variable satisfying~\eqref{suportW}, \eqref{normW} and~\eqref{limitWW}. Then, under certain conditions on $W$, there exists a random variable $V$, independent of $W$ and with support in $[0,1]$, such that~\eqref{taabbs} holds. Note that, from
\begin{equation} -\log W-\log V=-\log U\sim\mathrm{Exponential}(1), \label{djsfkoeokks} \end{equation}
we essentially need to decompose $\mathrm{Exponential}(1)$ as a sum of non-negative and independent random variables, where 
\[ \tilde{W}=-\log W \]
is previously fixed with the properties 
\[ \mathbb{P}[\tilde{W}\in [0,\epsilon)]>0 \text{ for all }\epsilon>0, \]
where $\mathbb{P}$ is the probability measure, and 
\[ \lim_{n\rightarrow\infty} n\mathbb{E}[\mathrm{e}^{-n \tilde{W}}]=0. \]
\end{conjecture}

\textit{Comments on Conjecture~\ref{conj_WWW}:} Note that~\eqref{taabbs} implies that 
\[ \mathbb{E}[W^r]\mathbb{E}[V^r]=\mathbb{E}[U^r]=\frac{1}{1+r}, \]
i.e.,
\begin{equation} \mathbb{E}[V^r]=\frac{1}{(1+r)\mathbb{E}[W^r]}. \label{momeVVV} \end{equation}
The question is whether there exists $V$ with those specific moments~\eqref{momeVVV}. On the other hand, working with~\eqref{djsfkoeokks} and characteristic functions, the point is whether
\begin{equation} \varphi(u)=\frac{\varphi_{\mathrm{Ex}(1)}(u)}{\varphi_{\tilde{W}}(u)} \label{ratte} \end{equation}
is a characteristic function of a non-negative random variable. Finally, dealing with~\eqref{djsfkoeokks} and density functions, the question is whether the convolution
\[ (\rho_{\tilde{W}}\ast \rho_{-\log V})(t)=\rho_{\mathrm{Ex}(1)}(t)=\mathrm{e}^{-t},\quad t>0, \]
is invertible, i.e., computing the deconvolution of functions.

\begin{Example} \label{ex_m_isk} \normalfont
Let us see a positive example for Conjecture~\ref{conj_WWW}, distinct to the L-fractional context. Consider 
\[ \tilde{W}=-\log W\sim\mathrm{Gamma}(\alpha,1) \]
and 
\[ \tilde{V}=-\log V\sim\mathrm{Gamma}(1-\alpha,1), \]
independent, where $0<\alpha<1$ is the shape and $1$ is the rate. Then~\eqref{djsfkoeokks} holds. Besides, for 
\begin{equation} W=\mathrm{e}^{-\tilde{W}}\sim \mathrm{e}^{-\mathrm{Gamma}(\alpha,1)}, \label{tenc_bril} \end{equation}
condition~\eqref{limitWW} is true, from the moment-generating function of the gamma distribution:
\begin{align*}
 \lim_{n\rightarrow\infty} n\mathbb{E}[W^n]= {} & \lim_{n\rightarrow\infty} n\mathbb{E}[\mathrm{e}^{-n\tilde{W}}] \\
= {} & \lim_{n\rightarrow\infty} \frac{n}{(1+n)^\alpha}=\infty.
\end{align*}
In conclusion, $W$ defined by~\eqref{tenc_bril} is a suitable random variable to define a fractional operator~\eqref{new_fr_derii}, of one parameter $\alpha$, and its used in modeling shall be investigated. When $\alpha\rightarrow 1^-$, the ordinary derivative is obtained. As now will be seen,~\eqref{new_fr_derii_J2} becomes its inverse integral operator $\mathcal{J}$. From the densities
\[ \rho_W(w)=\rho_{\tilde{W}}(-\log w)\frac{1}{w}=\frac{1}{\Gamma(\alpha)}(-\log w)^{\alpha-1}, \]
the explicit expressions of the fractional operators are the following:
\begin{align*}
 \mathcal{D}x(t)= {} & \frac{1}{\Gamma(\alpha)t}\int_0^t \left(-\log \frac{s}{t}\right)^{\alpha-1}x'(s)\mathrm{d}s \\
= {} & \frac{1}{\Gamma(\alpha)t}\int_0^t \left(\log \frac{t}{s}\right)^{\alpha-1}x'(s)\mathrm{d}s
\end{align*}
and
\begin{align*}
 \mathcal{J}x(t)= {} & \frac{1}{\Gamma(1-\alpha)}\int_0^t \left(-\log \frac{s}{t}\right)^{-\alpha}x(s)\mathrm{d}s \\
= {} & \frac{1}{\Gamma(1-\alpha)}\int_0^t \left(\log \frac{t}{s}\right)^{-\alpha}x(s)\mathrm{d}s. 
\end{align*}
\end{Example}

\begin{Example} \label{ex_m_isk2} \normalfont
Let us see another valid distribution for Conjecture~\ref{conj_WWW}, distinct to the gamma. Consider
\[ \tilde{W}=-\log W\sim\mathrm{Bernoulli}(1/2) \cdot \mathrm{Exponential}(1) \]
and
\[ \tilde{V}=-\log V\sim \frac12 \mathrm{Exponential}(1), \]
independent. Recall that $X\sim \mathrm{Bernoulli}(1/2)$ if it takes the values $0$ and $1$ with probability $1/2$, respectively, and $Y\sim \mathrm{Exponential}(1)$ if its density function is $\mathrm{e}^{-y}$. Observe that $\tilde{W}$ is not absolutely continuous and its density function is expressed in terms of the Dirac delta function. The moment-generating functions are
\[ \mathbb{E}[\mathrm{e}^{\tilde{W}u}]=\frac12 \frac{1}{1-u}+\frac{1}{2}=\frac{2-u}{2-2u} \]
and
\[ \mathbb{E}[\mathrm{e}^{\tilde{V}u}]=\frac{1}{1-u/2}, \]
therefore
\[ \mathbb{E}[\mathrm{e}^{(\tilde{W}+\tilde{V})u}]=\frac{2-u}{2-2u}\times \frac{1}{1-u/2}=\frac{1}{1-u}, \]
that is,~\eqref{djsfkoeokks} holds. Note that
\begin{align*}
 \lim_{n\rightarrow\infty} n\mathbb{E}[W^n]= {} & \lim_{n\rightarrow\infty} n\mathbb{E}[\mathrm{e}^{-n\tilde{W}}] \\
= {} & \lim_{n\rightarrow\infty} n\frac{2+n}{2+2n}=\infty,
\end{align*}
so~\eqref{limitWW} is verified. As now will be seen,~\eqref{new_fr_derii_J2} is the inverse operator $\mathcal{J}$ of~\eqref{new_fr_derii}. The explicit expressions of the operators are
\begin{align*}
 \mathcal{D}x(t)= {} & \mathbb{E}[x'(t\mathrm{e}^{-\tilde{W}})] \\
 = {} & \frac12 \mathbb{E}[x'(tU)]+\frac12 x'(t)=\frac12 \frac{x(t)-x(0)}{t}+\frac{1}{2}x'(t)
\end{align*}
and
\begin{align*}
 \mathcal{J}x(t)= {} & t\mathbb{E}[x(t\mathrm{e}^{-\tilde{V}})] \\
 = {} & \int_0^t x(s)\rho_{V}(s/t)\mathrm{d}s=\frac{2}{t}\int_0^t sx(s)\mathrm{d}s. 
\end{align*}
The interpretation of $\mathcal{D}x(t)$ in this case is very simple in terms of $x'$: the average between the historical mean value $\frac{1}{t}\int_0^t x'(s)\mathrm{d}s$ and the value at the present time $t$. 
\end{Example}

\begin{Example} \label{ex_m_isk33} \normalfont
We present a negative example for Conjecture~\ref{conj_WWW}, which shows that certain conditions on $W$ are required in general for the existence of $V$. Let
\begin{equation} \tilde{W}=-\log W\sim 2\cdot \mathrm{Bernoulli}(1/2) \cdot \mathrm{Exponential}(1). \label{tWWw} \end{equation}
Then
\[ \mathbb{E}[\mathrm{e}^{\tilde{W}u}]=\frac12 \frac{1}{1-2u}+\frac12=\frac{1-u}{1-2u}, \]
so
\begin{align*}
 \lim_{n\rightarrow\infty} n\mathbb{E}[W^n]= {} & \lim_{n\rightarrow\infty} n\mathbb{E}[\mathrm{e}^{-n\tilde{W}}] \\
= {} & \lim_{n\rightarrow\infty} n\frac{1+n}{1+2n}=\infty,
\end{align*}
and~\eqref{limitWW} holds. Now, we have the following ratio~\eqref{ratte} of characteristic functions:
\begin{align*}
 \varphi(u)= {} & \frac{\varphi_{\mathrm{Ex}(1)}(u)}{\varphi_{\tilde{W}}(u)} \\
= {} & \frac{\frac{1}{1-u\mathrm{i}}}{\frac{1-u\mathrm{i}}{1-2u\mathrm{i}}} \\
= {} & \frac{1-2u\mathrm{i}}{(1-u\mathrm{i})^2} \\
= {} & \frac{1-2u\mathrm{i}}{1-u^2-2u\mathrm{i}} \\
= {} & \frac{(1-2u\mathrm{i})(1-u^2+2u\mathrm{i})}{(1-u^2-2u\mathrm{i})(1-u^2+2u\mathrm{i})} \\
= {} & \frac{1+3u^2+2u^3\mathrm{i}}{(1+u^2)^2},
\end{align*}
where $\mathrm{i}$ is the imaginary unit. Then,
\[ |\varphi(u)|=\frac{\sqrt{(1+3u^2)^2+4u^6}}{(1+u^2)^2}. \]
However, numerical computations show that
\[ |\varphi(0.1)|\approx 1.00971,\quad |\varphi(0.2)|\approx 1.03561,\quad |\varphi(0.7)|\approx 1.15467. \]
This is impossible for a characteristic function, because its modulus should be less than or equal to $1$ for every $u\in\mathbb{R}$. Therefore, there is no random variable $\tilde{V}$, independent of $\tilde{W}$, that meets $\tilde{W}+\tilde{V}\sim\mathrm{Exponential}(1)$, i.e.,~\eqref{djsfkoeokks}. We emphasize that it is not necessary for $\tilde{W}$ to have a discrete part; we can take an absolutely continuous random variable $\tilde{W}_2$ that is sufficiently near to~\eqref{tWWw} in distribution, such that $\varphi_{\tilde{W}_2}(u)$ tends to $\varphi_{\tilde{W}}(u)$ (by L\'evy's continuity theorem \cite[section~18.1]{willia}), and still have $|\varphi(u)|>1$. This example is finished.
\end{Example}

By taking~\eqref{proba_L} into account, we define the associated integral operator to~\eqref{new_fr_derii} as
\begin{equation}
 \mathcal{J}x(t)=t\mathbb{E}[x(tV)],
 \label{new_fr_derii_J2}
\end{equation}
where $t\in [0,T]$, $x:[0,T]\rightarrow\mathbb{C}$ is a continuous function, and $V$ is a random variable, independent of $W$ and with support in $[0,1]$, that satisfies~\eqref{taabbs} (we assume this condition, as in the L-fractional case or Examples~\ref{ex_m_isk} and~\ref{ex_m_isk2}, for instance; in general, we would need Conjecture~\ref{conj_WWW}). 

With~\eqref{new_fr_derii_J2}, there is a version of the fundamental theorem of calculus.

\begin{Theorem} \label{ftcdcc}
If $x:[0,T]\rightarrow\mathbb{C}$ is a continuously differentiable function, then
\begin{equation} \mathcal{J}\circ \mathcal{D}x(t)=x(t)-x(0) \label{JDx_nou} \end{equation}
and
\begin{equation} \mathcal{D}\circ\mathcal{J}x(t)=x(t) \label{JDx_nou2} \end{equation}
for every $t\in [0,T]$. If $x$ is given by a power series~\eqref{sumsid}, then
\begin{equation} \mathcal{J}x(t)=\sum_{n=0}^\infty x_n \mathcal{J}t^n=\sum_{n=0}^\infty x_n \frac{t^{n+1}}{(n+1) \mathbb{E}[W^{n}]}. \label{JDX_ps} \end{equation}
\end{Theorem}
\begin{proof}
First, we note that $\mathcal{D}x\in\mathcal{C}[0,T]$ and $\mathcal{J}x\in\mathcal{C}^1[0,T]$. Indeed, if $t_1\rightarrow t_2$ in $[0,T]$, then $x'(t_1W)\rightarrow x'(t_2W)$, and since $\|x'\|_\infty<\infty$ on $[0,T]$, the dominated convergence theorem gives $\mathcal{D}x(t_1)\rightarrow \mathcal{D}x(t_2)$. For $\mathcal{J}x\in\mathcal{C}^1[0,T]$, just note that $t\mapsto \mathbb{E}[x(tV)]$ is $\mathcal{C}^1[0,T]$, because 
\[ \left|\frac{\mathrm{d}}{\mathrm{d}t} x(tV)\right|=|Vx'(tV)|\leq \|x'\|_\infty, \]
which has finite expectation; the dominated convergence theorem is employed to ensure that 
\[ \frac{\mathrm{d}}{\mathrm{d}t}\mathbb{E}[x(tV)]=\mathbb{E}\left[\frac{\mathrm{d}}{\mathrm{d}t} x(tV)\right]. \]
Thus, the conclusion of this remark is that both compositions in~\eqref{JDx_nou} and~\eqref{JDx_nou2} are well defined when $x\in\mathcal{C}^1[0,T]$. Now we aim at proving the equalities in~\eqref{JDx_nou} and~\eqref{JDx_nou2}.

For~\eqref{JDx_nou}, by using the properties of the conditional expectation,
\begin{align*}
 \mathcal{J}\circ \mathcal{D}x(t)= {} & t\mathbb{E}[\mathcal{D}x(Vt)] \\
= {} & t\mathbb{E}[\mathbb{E}_{W}[x'(WVt)]] \\
= {} & t \mathbb{E}[x'(WVt)] \\
= {} & t\mathbb{E}[x'(Ut)] \\
= {} & t\int_0^1 x'(ut)\mathrm{d}u \\
= {} & \int_0^t x'(s)\mathrm{d}s \\
= {} & x(t)-x(0).
\end{align*}

On the other hand, for~\eqref{JDx_nou2},
\begin{align*}
\mathcal{D}\circ\mathcal{J}x(t)= {} & \mathcal{D}[t\mathbb{E}_{V}[x(tV)]] \\
= {} & \mathbb{E}[ \mathbb{E}_{V}[x(tVW)]] + t\mathbb{E}[W\mathbb{E}_{V}[Vx'(tVW)]] \\
= {} & \mathbb{E}[x(tU)]+t\mathbb{E}[Ux'(tU)] \\
= {} & \int_0^1 x'(ut)\mathrm{d}u + t \int_0^1 ux'(ut)\mathrm{d}u \\
= {} & \frac{1}{t}\int_0^t x(s)\mathrm{d}s + \frac{1}{t}\int_0^t sx'(s)\mathrm{d}s \\
= {} & \frac{1}{t}\int_0^t x(s)\mathrm{d}s + \frac{1}{t}\left(tx(t)-\int_0^t x(s)\mathrm{d}s\right) \\
= {} & x(t).
\end{align*}
Integration by parts has been used in the last line.

Finally, for~\eqref{JDX_ps}, observe that $\mathcal{J}$ is a continuous operator from $\mathcal{C}[0,T]$ into $\mathcal{C}[0,T]$: if 
\[ \|x-y\|_\infty\rightarrow0, \]
then 
\[ \|\mathcal{J}x-\mathcal{J}y\|_\infty \leq T\|x-y\|_\infty\rightarrow0. \]
Hence~\eqref{JDX_ps} holds, considering that
\[ \mathcal{J}t^n=t\mathbb{E}[t^nV^n]=t^{n+1}\frac{\mathbb{E}[U^n]}{\mathbb{E}[W^n]}=\frac{t^{n+1}}{(n+1)\mathbb{E}[W^n]}. \]
\end{proof}

We give a version of Picard's theorem for~\eqref{ode_DDDD}, based on fixed-point theory. 

\begin{Theorem} \label{ara_escds22}
If $f:[0,T]\times\Omega\subseteq [0,T]\times\mathbb{R}^d\rightarrow \mathbb{R}^d$ is continuous and is locally Lipschitz continuous with respect to the second variable (i.e., for every $(t_1,x_1)\in [0,T]\times\Omega$, it is Lipschitz on a neighborhood of it), then the integral problem associated to the differential equation~\eqref{ode_DDDD}, 
\begin{equation} x(t)=x_0+\mathcal{J}f(t,x(t))=x_0+t\mathbb{E}[f(tV,x(tV))], \label{nananadas} \end{equation}
presents local existence and uniqueness of solution in the set of continuous functions. Specifically, if 
\[ Q=[0,a]\times \overline{B}(x_0,b)\subseteq [0,T]\times\Omega \]
is a rectangle where $f$ is $M$-Lipschitz and $\overline{B}(x_0,b)$ is the closed ball of $\|\cdot\|_\infty$-radius $b$ centered at $x_0$, then an interval of definition of the solution is $[0,T^\ast]$, such that $T^\ast\leq a$ and
\[ T^\ast<\min\left\{\frac{b}{\|f\|_\infty},\frac{1}{M}\right\}, \]
where $\|f\|_\infty$ is the maximum of $f$ on $Q$. Finally, if the solution of~\eqref{nananadas} is continuously differentiable, then it solves~\eqref{ode_DDDD} locally, by Theorem~\ref{ftcdcc}.
\end{Theorem}
\begin{proof}
Consider the Banach space
\[ \Theta=\{y\in\mathcal{C}[0,T^\ast]:\,\|y-x_0\|_\infty\leq b\}. \]
Let
\[ \Lambda:\Theta\rightarrow\Theta, \]
\[ \Lambda x(t)=x_0+t\mathbb{E}[f(tV,x(tV))]. \]

Given $x\in\Theta$, if $t_1\rightarrow t_2$ in $[0,T^\ast]$, then
\[ f(t_1V,x(t_1V))\rightarrow f(t_2V,x(t_2V)) \]
almost surely by continuity, and since $f$ is bounded on $[0,T]\times x([0,T^\ast])$, the dominated convergence theorem ensures that
\[ \Lambda x(t_1)\rightarrow \Lambda x(t_2), \]
that is, $\Lambda x\in\mathcal{C}[0,T^\ast]$. On the other hand, if $t\in [0,T^\ast]$, then
\[ |\Lambda x(t)-x_0|\leq t \|f\|_\infty \leq b, \]
therefore
\[ \|\Lambda x-x_0\|_\infty\leq b \]
and $\Lambda x\in\Theta$. Hence $\Lambda$ is well defined.

Let us see that $\Lambda$ is a contraction: if ${}_1 x,{}_2 x\in\Theta$, then
\begin{align*}
 |\Lambda [{}_1 x](t)-\Lambda [{}_2 x](t)|\leq {} & t\mathbb{E}[|f(tV,{}_1 x(tV))-f(tV,{}_2 x(tV))|] \\
 \leq {} & M t\mathbb{E}[|{}_1 x(tV) - {}_2 x(tV)|] \\
\leq {} & M T^\ast \|{}_1 x - {}_2 x\|_\infty.
\end{align*}
Since $M T^\ast<1$, $\Lambda$ is a contraction and, by Banach fixed-point theorem \cite[Chapter~1]{granas}, it has a unique fixed point $x\in\Theta$.
\end{proof}

\subsection{Power-series solutions of nonlinear equations: The Cauchy-Kovalevskaya theorem}

We work in the context of Sections~\ref{subs_def_WW} and~\ref{subs_pankds}. We do not need the integral operator $\mathcal{J}$ from the previous part.

We focus on scalar problems with polynomial vector field:
\begin{equation} \mathcal{D} x=\sum_{i=0}^m a_i x^i, \label{rhsejak} \end{equation}
where $a_0,\ldots,a_m\in\mathbb{R}$. For example, a logistic vector field would be $\mu x(1-x/K)$, $\mu\in\mathbb{R}$ and $K>0$, with quadratic nonlinearity $m=2$, which has previously been studied in the Caputo sense~\cite{logistic_nieto}.

\begin{Theorem} \label{th_primer_ck}
Equation~\eqref{rhsejak} has a power-series solution, that converges on a neighborhood $[0,\epsilon)$.
\end{Theorem}
\begin{proof}
By using a power series of the form~\eqref{sumsid} and Proposition~\ref{propisa}, we have
\[ \sum_{n=0}^\infty x_{n+1}(n+1)\mathbb{E}[W^n]t^n=\sum_{n=0}^\infty \left( \sum_{i=0}^m a_i C_i(x_0,\ldots,x_n)\right)t^n, \]
where $C_i(x_0,\ldots,x_n)$ are $n$-th terms of Cauchy products of power $i$. The coefficients then satisfy
\[ x_{n+1}=\frac{1}{(n+1)\mathbb{E}[W^n]}\sum_{i=0}^m a_i C_i(x_0,\ldots,x_n),\quad x_0=x(0), \]
where $C_i(x_0,\ldots,x_n)$ are $n$-th terms of Cauchy products of power $i$. We know that there exists a constant $K>0$ such that
\[ \frac{1}{(n+1)\mathbb{E}[W^n]}\leq K \]
for all $n\geq0$, by the stronger condition~\eqref{limitWW}. Then, by the triangular inequality,
\[ |x_{n+1}|\leq K\sum_{i=0}^m |a_i| C_i(|x_0|,\ldots,|x_n|). \]
The formal majorizing series $\psi(t)=\sum_{n=0}^\infty y_n t^n$ has coefficients
\[ y_{n+1}=C\sum_{i=0}^m |a_i| C_i(y_0,\ldots,y_n),\quad y_0=|x_0|. \]
We have the following functional, algebraic equation:
\[ \psi(t)=y_0+t\sum_{n=0}^\infty y_{n+1}t^n=y_0+Kt\sum_{i=0}^m |a_i| \psi(t)^i. \]
With 
\[ \phi(t,w)=w-y_0-Kt\sum_{i=0}^m |a_i| w^i,\quad (t,w)\in\mathbb{R}^2, \]
one has $\phi(0,y_0)=0$ and $\frac{\partial \phi}{\partial w}(0,y_0)=1\neq0$. The implicit-function theorem  \cite[Section~8, Chapter~0]{liib} shows that there exists a unique analytic function $w\equiv w(z)$ on a neighborhood of zero such that $w(0)=y_0$ and $\phi(z,w(z))=0$. This implies that $w(z)=\psi(z)$ is analytic at $0$.
\end{proof}

\begin{Example} \label{ex_SiR_eq} \normalfont
Fractional differential equations have been used to extend customary epidemic models~\cite{area_sirr}. We consider a generalization of the SIR (susceptible-infected-recovered) model,
\begin{equation}
 \begin{cases} \mathcal{D} S=-\beta SI, \\ \mathcal{D} I=\beta SI-\gamma I, \\ \mathcal{D} R=\gamma I, \end{cases} 
 \label{sir_clas_L}
\end{equation}
where $\beta>0$ and $\gamma>0$ control the infection and the recovery rates, respectively. With the form~\eqref{sumsid} and the differentiation rule from Proposition~\ref{propisa}, one has
\[ \mathcal{D} S(t)=\sum_{n=0}^\infty s_{n+1}(n+1)\mathbb{E}[W^n]t^{n}, \]
\[ \mathcal{D} I(t)=\sum_{n=0}^\infty i_{n+1}(n+1)\mathbb{E}[W^n]t^{n}, \]
\[ \mathcal{D} R(t)=\sum_{n=0}^\infty r_{n+1}(n+1)\mathbb{E}[W^n]t^{n}. \]
On the other hand, we have the Cauchy product
\[ S(t)I(t)=\sum_{n=0}^\infty \left(\sum_{m=0}^n s_m i_{n-m}\right) t^n. \]
Then,
\begin{equation} s_{n+1}=-\frac{1}{(n+1)\mathbb{E}[W^n]}\beta\sum_{k=0}^n s_k i_{n-k}, \label{sn1} \end{equation}
\begin{equation} i_{n+1}=\frac{1}{(n+1)\mathbb{E}[W^n]}\left[\beta \sum_{k=0}^n s_k i_{n-k}-\gamma i_n\right], \label{in1} \end{equation}
\begin{equation} r_{n+1}=\frac{1}{(n+1)\mathbb{E}[W^n]}\gamma i_n \label{rn1} \end{equation}
are the coefficients of the solution. The proof of convergence with \eqref{sn1}--\eqref{rn1} is very similar to Theorem~\ref{th_primer_ck}, now adapted to the case of a system of three equations~\cite{sir_nie}. Based on~\eqref{limitWW}, we bound
\[ |s_{n+1}|\leq C\beta \sum_{k=0}^n |s_k||i_{n-k}|, \]
\[ |i_{n+1}|\leq C\beta \sum_{k=0}^n |s_k||i_{n-k}| + C\gamma|i_n|, \]
\[ |r_{n+1}|\leq C\gamma |i_n|, \]
where $C>0$ is a constant. Then we consider the majorizing sequences $\tilde{s}_0=|s_0|$, $\tilde{i}_0=|i_0|$, $\tilde{r}_0=|r_0|$,
\[ \tilde{s}_{n+1}=C\beta \sum_{k=0}^n \tilde{s}_k \tilde{i}_{n-k}, \]
\[ \tilde{i}_{n+1}=C\beta \sum_{k=0}^n \tilde{s}_k \tilde{i}_{n-k}+C\gamma\tilde{i}_n, \]
\[ \tilde{r}_{n+1}=C\gamma \tilde{i}_n. \]
If we define the formal power series
\[ \tilde{S}(t)=\sum_{n=0}^\infty \tilde{s}_n t^n, \]
\[ \tilde{I}(t)=\sum_{n=0}^\infty \tilde{i}_n t^n, \]
\[ \tilde{R}(t)=\sum_{n=0}^\infty \tilde{r}_n t^n, \]
then
\[ \tilde{S}(t)=|s_0|+Ct\beta \tilde{S}(t)\tilde{I}(t), \]
\[ \tilde{I}(t)=|i_0|+Ct(\beta \tilde{S}(t)\tilde{I}(t)+\gamma\tilde{I}(t)),\]
\[ \tilde{R}(t)=|r_0|+Ct\gamma\tilde{I}(t). \]
The analytic function
\[ \phi:\mathbb{R}^4\rightarrow\mathbb{R}^3, \]
\[ (t,x,y,z)\mapsto (x-|s_0|-Ctxy,y-|i_0|-Ct(\beta xy+\gamma y),z-|r_0|-Cty) \]
satisfies
\[ \phi(0,\tilde{s}_0,\tilde{i}_0,\tilde{r}_0)=0, \]
\[ J\phi(0,\tilde{s}_0,\tilde{i}_0,\tilde{r}_0)=1, \]
where $J$ is the Jacobian with respect to $(x,y,z)$. By the implicit-function theorem, there are analytic functions $x(t)$, $y(t)$ and $z(t)$ on a neighborhood of $t=0$ such that $\phi(t,x(t),y(t),z(t))=0$ and $x(0)=\tilde{s}_0$, $y(0)=\tilde{i}_0$, $z(0)=\tilde{r}_0$. Then $\tilde{S}=x$, $\tilde{I}=y$ and $\tilde{R}=z$ are analytic functions, as wanted. Once this model~\eqref{sir_clas_L} is well-defined, further extensions of it could be investigated, such as the incorporation of a Brownian perturbation giving rise to a stochastic fractional differential equation.
\end{Example}

\section{Conclusions and open problems} \label{sec_conc}

The main novelties of the presented paper were:
\begin{itemize}
\item The definition of normalized fractional operators with non-singular kernel, for the first time, in Section~\ref{subsec2_norm}. In the literature, there have been documented deficiencies of fractional operators with bounded kernels, see~\cite{dieth}. Our work showed that a rescaling fixes the issues of these operators: inconsistency at zero, units time$^0$, and lack of fundamental theorem of calculus. Our new operators and equations are mathematically valid and could be further investigated, beyond exponential and Mittag-Leffler kernels. During the time the present paper has been in the preprint server ArXiv and under review, new works citing my suggested normalized operators are being conducted~\cite{balen_nou}.
\item The definition of a general class of fractional operators, based on a probabilistic approach, in Section~\ref{new_sec_memo}. My previous articles~\cite{arx_jo} and~\cite{aml_jo} dealt with the L-fractional derivative, defined as the normalization of the Caputo derivative. The L-fractional derivative was relevant, as it offered alternative properties: units time$^{-1}$ instead of time$^{-\alpha}$, smoothness of solutions, finite ordinary derivative at the origin, etc. These properties were reviewed in Section~\ref{subsec1_norm}. The novel Section~\ref{new_sec_memo} generalized, for the first time, the L-fractional derivative and the normalization of the Caputo operator via probability theory, by defining fractional operators with an averaged probabilistically distributed past. Many properties were stated and demonstrated: the validity and consistency of the definition, the associated Mittag-Leffler function, existence and uniqueness of solution by fixed-point theory, and an example related to the SIR model.
\end{itemize}

Besides the future research lines suggested in~\cite{arx_jo}, some open problems from the present paper are the following:
\begin{itemize}
\item The study of more properties of rescaled operators with bounded kernels. See Section~\ref{subsec2_norm}.
\item The development of more theory on operators with memory~\eqref{new_fr_derii}. Specifically, it would be of great relevance to obtain a complete resolution of Conjecture~\ref{conj_WWW}. This would better characterize when the fundamental theorem of calculus and existence-and-uniqueness results hold; see Section~\ref{subsec_conjesdz}.
\item The study of the new Mittag-Leffler-type function~\eqref{mlf_kon}: representation formulas, dynamics, asymptotic values, etc. 
\item The investigation of dynamical systems based on the new operator~\eqref{new_fr_derii}.
\item The design of numerical methods for~\eqref{ode_DDDD}. The search of applications in modeling.
\end{itemize}

\section*{Funding}
This work received no funding.

\section*{Data Availability Statement}
No data were used for this study.

\section*{Conflict of interest}
The author declares that there is no conflict of interest regarding the publication of this article.

\end{document}